\documentclass[12pt]{article}
 \usepackage{setspace}
 \usepackage[all,cmtip]{xy}
\usepackage{enumerate}
\usepackage{ amssymb, latexsym, amsmath}

\usepackage{lingmacros}
\usepackage{tree-dvips}
\usepackage{amssymb,amsfonts,amsthm,amsmath}

\usepackage{amsmath,amssymb,amsfonts,amscd}
\usepackage[utf8]{inputenc}
\usepackage{amsmath}
\usepackage{amssymb}
\usepackage[hidelinks]{hyperref}
\usepackage[margin=0.9in]{geometry}
\usepackage{parskip}
\usepackage{xcolor}
\newtheorem{theorem}{Theorem}
\newtheorem{proposition}{Proposition}

\newtheorem{corollary}{Corollary}
\newtheorem{defn}{Definition}
\newtheorem{example}{Example}
\newtheorem{remark}{Remark}
\newtheorem*{theorem*}{Theorem}

\usepackage{setspace}
\makeatletter
\def\thm@space@setup{%
  \thm@preskip=8pt
  \thm@postskip=8pt
}
\makeatother
\begin{document}

\title{Existence of Reeb Chords from the Generation-Criteria}
\author{Mohamad Rabah }
\date{}

\maketitle

\begin{abstract}
    We give some existence results of Reeb chords using the $\textit{generation-criteria}$ of the Fukaya category. In particular, for $n\geq3$, we construct a Legendrian $\ell'\subset S^*\mathbb T^n$ such that every closed spin Legendrian $\ell\subset S^*\mathbb T^n$ either admits a self-Reeb chord or a Reeb chord joining it to $\ell'$.
\end{abstract}

\section{Introduction}
We study the existence of Reeb chords from the perspective of Lagrangian Floer theory. Recall that Arnol'd's chord conjecture predicts that every closed Legendrian submanifold of a closed contact manifold admits a non-constant Reeb chord for every contact form defining the given contact structure.
Our results establish variants of this principle in several symplectic settings and also allow chords connecting a given Legendrian to a distinguished Legendrian link arising from a collection of Lagrangians that $\textit{split-generate}$ the Fukaya category. We will prove such principle in various symplectic settings, with the main novelty lying in the closed case. To be more precise, we prove the following results:  

\begin{theorem*}
    Let $(X, \omega)$ be a closed symplectic manifold and let $\mathcal A=\langle(L_1, b_1), \dots, (L_k, b_k)\rangle$ be a finite-collection of weakly unobstructed relatively-spin Lagrangians satisfying the generation-criteria as in \cite{abouzaid2026quantum}. Fix a quantitative generation-criteria as in $\S32$ of \cite{abouzaid2026quantum} and denote its energy-shift by $\rho\geq0$. Let $V\subset X$ be a connected hypersurface of contact-type, intersecting $L_i$ in a Legendrian $\ell_i$. Let $\ell\subset V$ be any closed relatively-spin Legendrian and suppose that there exists a loop $\gamma$ used in Mohnke's Lagrangian construction as in \cite{mohnke2001holomorphic} such that $\delta_\gamma>\rho$, where $\delta_\gamma$ is as in Definition \ref{QuanCriteria}. Then, either $\ell$ admits a non-constant self-Reeb chord or there exists a Reeb chord joining $\ell$ to $\bigcup_{i=1}^k\ell_i$.
\end{theorem*}

As a consequence of the above Theorem, we show the following:

\begin{theorem*}[Corollary \ref{MohnkeResult}]
    In the setting as above and under the assumption that $(X, \omega)$ is symplectically aspherical and $\pi_1(V)\hookrightarrow\pi_1(X)$ is an injective group morphism. Every closed relatively-spin Legendrian $\ell$ of $V$ either admits a non-constant self-Reeb chord or there exists a Reeb chord joining $\ell$ to $\bigcup_{i=1}^k\ell_i$.
\end{theorem*}

Unlike Mohnke's Theorem as found in \cite{mohnke2001holomorphic}, our result does not require any Hamiltonian-displaceability assumption and in return, can be applied for instance, to the case when $X=\mathbb T^{2n}\cong\mathbb T^n\times \mathbb T^n$ for $n\geq 3$, together with its standard symplectic form, and $V=S^*\mathbb T^n$ the unit cosphere-bundle of the first factor $n$-torus. Indeed, we show the following:

\begin{theorem*}[Theorem \ref{cosphereResult}]
    For $n\geq3$, let $V=S^*\mathbb T^n$ be the unit cosphere-bundle of the flat torus $\mathbb T^n$ together with its canonical contact form $\alpha$ induced from $T^*\mathbb T^n$. There exists a closed embedded Legendrian link $\ell'$ such that, every closed spin Legendrian $\ell$ of $V$ either admits a non-constant self-Reeb chord for $\alpha$ or a Reeb chord for $\alpha$ joining $\ell$ to $\ell'$.
\end{theorem*}

\subsection{Organization and proof strategy}

In Sections $2$ and $3$, we deal with the case when the symplectic manifold is a Liouville manifold or monotone and convex at infinity, respectively. The arguments we present are based on a common contradiction argument using the wrapped Floer cohomology. Assuming that the relevant Legendrian carries no Reeb chords, it follows that, there are no Hamiltonian chords in the cylindrical end. Hence, the wrapped Floer complex remains finite-dimensional. This contradicts generation-criteria forcing either triviality of the wrapped floer cohomology or being of infinite rank.

The key novelty in this note is in Section $4$ where we deal with the case of a general closed symplectic manifold. We apply a truncated version of the illustrated contradiction argument above, using Lagrangian Floer theory of \cite{Fukaya2009} and the generation criterion of \cite{abouzaid2026quantum}.

\subsection*{Acknowledgments}
I would like to thank Professors Kaoru Ono, Egor Shelukhin, and Zhengyi Zhou for their interest in this work; Hang Yuan and Yao Xiao for helpful discussions. I would also like to thank my girlfriend, Tracy Zhang Ke, for her support and encouragement throughout the preparation of this work. Finally, I am deeply grateful to my teachers, Professors Kenji Fukaya and Mark McLean; much of my mathematical knowledge is the outcome of their patience and guidance.

\section{Liouville Manifold Case}
	Suppose that $(M,\omega)$ is a non-degenerate Liouville manifold thought of as the completion of the Liouville domain $(\overline{M}, \lambda)$ with a its wrapped Fukaya category denoted by $\mathcal W(M)$, defined over $\mathbb Z/2$. Let $L_1,\dots, L_k$ be a finite-collection of properly embedded exact cylinderical Lagrangians that split-generates $\mathcal{W}(M)$. Denote by $\ell_i:=L_i\cap \partial\overline{M}$ and set $\ell:=\bigcup_{i=1}^k\ell_i$.
	\begin{proposition}
		The Legendrian link $\ell$ has a Reeb chord for the contact form $\alpha\equiv \lambda|_{\partial\overline{M}}$.
	\end{proposition}
	\begin{proof}
    If $(\overline{M}, \lambda)$ has a trivial wrapped Fukaya category $\mathcal W(M)$ then, $HW^*(L_i, L_i)=\{0\}$ for any $i=1,\dots, k$. On the other hand, if $\ell_i$ admits no Reeb chord for $\alpha$ then, $HW^*(L_i, L_i)=H^*(W; \mathbb Z/2)\neq\{0\}$, which is absurd. Now assume that $\mathcal{W}(M)$ is non-trivial and assume for a contradiction that for every $i,j=1,\dots, k$, there exists no Reeb chords from $\ell_i$ to $\ell_j$. Then, $CW^*(L_i, L_j)$ is finitely-generated and hence $HW^*(L_i, L_j)$ has finite rank. As $L_1, \dots, L_k$ split-generate $\mathcal W(M)$, it follows that $\mathcal{W}(M)$ is proper, contradicting the main Theorem of \cite{ganatra2021categorical}. Thus, there exists some $i,j=1,\dots, k$ such that $CW^*(L_i, L_j)$ is not finitely-generated and in particualr, there are infinitely-many Reeb chords from $\ell_i$ to $\ell_j$ as desired.
	\end{proof}
    Now suppose that $\ell'\subset\partial\overline{M}$ is exact fillable. That is, there exists an exact cylinderical Lagrangian $K$ such that $K\cap\partial\overline{M}=\ell'$. Once again, using the main Theorem of \cite{ganatra2021categorical}, we can deduce the following existence result. 
	\begin{corollary}
		$\ell'$ admits a self Reeb chord or there exists a Reeb chord from $\ell'$ to the Legendrian link $\ell$. 
	\end{corollary}
	\begin{proof}
		Suppose that $\ell'$ admits no self Reeb chords then, $CW^*(K,K)$ is finitely-generated and $HW^*(K, K)\cong H^*(K; \mathbb Z/2)\neq\{0\}$. Hence, $K\neq 0\in\mathcal W(M)$. Thus, if $\mathcal W(M)$ is trivial, then $HW^*(K, K)=\{0\}$ which is absurd. Now suppose that $\mathcal W(M)$ is non-trivial and suppose in addition that we do not have any Reeb chords from $\ell'$ to $\ell$ then, generically $CW^*(K, L_i)$ is finitely-generated for every $i=1,\dots, k$. Therefore, $K$ is a left-proper object in $\mathcal W(M)$, contradicting the main Theorem of \cite{ganatra2021categorical}. 
	\end{proof}

\begin{remark}
    In fact, the proof above implies that in the case when $\mathcal W(M)$ is non-trivial, if $\ell'$ admits no self Reeb chords for $\alpha\equiv\lambda|_{\partial\overline{M}}$ then, there exists an $i=1,\dots, k$ and infinitely many Reeb chords connecting $\ell'$ to $\ell_i$
\end{remark}

\section{Monotone and Convex at Infinity Case}
	Suppose that $(E,\omega)$ is a monotone symplectic manifold which is convex at infinity and let $Y$ be a contact-type hypersurface with a
	 contact form $\alpha$.  Let $K\subset E$ be an admissible cylinderical monotone Lagrangian, as in $\S 3.4$ of \cite{ritter2017monotone}, such that, its cylinderical end is given by a cylinder over a closed Legendrian $\ell_K$ in $Y$. That is, $K\cap Y=\ell_K$.
	 
	 \begin{proposition}\label{MonotoneCriteria}
	 	Assume that:
	 	\begin{enumerate}
	 		\item $m_0(1)=\lambda$, that is, $K\in\mathcal{W}_\lambda(M)$ the $\lambda$-summand of $\mathcal W(M)$,
	 		\item the generalized $\lambda$-eigensummand $SH^*_\lambda(E)$ is trivial, and
	 		\item for $\epsilon>0$ sufficiently small and using an $\epsilon$-wrapping Hamiltonian $H_\epsilon$, $HF^*(K,K;H_\epsilon)$ is non-trivial.
	 	\end{enumerate}
	 	Then, $\ell_K$ admits a Reeb chord for $\alpha$.
	 \end{proposition}
	 \begin{proof}
	 	We will prove the case when $K$ is unobstructed, that is $\lambda\equiv0$, as the general case is matter of notational change for the following argument.\\
	 	
	 	Suppose for a contradiction, that $\ell_K$ admits no Reeb chords for $\alpha$. Choose a cofinal family of admissible wrapping Hamiltonians $\{H_n\}_{n\geq1}$ such that, $H_n\equiv H_\epsilon$ on a fixed non-empty compact subset away from the cylinderical end, for all $n\geq 1$. As $\ell_K$ admits no Reeb chords for $\alpha$, it follows that, $H_n$ admits no Hamiltonian chords in the cylinderical end. Therefore, for any $n\geq1$, all Hamiltonian chords of $H_n$ remain in the fixed compact subset. Note that, we are implicitly assuming that our almost-complex structure $J$ is of contact-type and hence, by an integrand Maximum Principle arguement, it follows that 
	 	\[
	 	HW^*(K,K)=HF^*(K,K; H_\epsilon)\neq\{0\}.\]
	 	Consider the Closed-Open map associated to $K$. Namely,
	 	\[
	 	\operatorname{CO}_K:SH^*(E)\rightarrow HW^*(K,K).
	 	\]
	 	By $\S 9.7$, more preciesely Theorem $9.6$ and Corollary $9.7$ of \cite{ritter2017monotone}, and under the assumption that $m_0(1)=0$, it follows that $\operatorname{CO}_K$ factors through $SH^*_0(E)$, the generalized zero-eigensummand. Under the assumption that $SH^*_0(E)=\{0\}$ and as $\operatorname{CO}_K$ is a unital ring morphism, it follows that $HW^*(K,K)=\{0\}$, contradicting our assumption.
	 \end{proof}
	 \begin{corollary}

	Suppose that $E:=\operatorname{Tot}(\mathcal{O}(-k)\rightarrow\mathbb{CP}^m)$ such that $1\leq k\leq m$ and $Y$ is its pre-quantization boundary. Let $K$ be an admissible monotone unobstructed cylinderical Lagrangian, with a Legendrian end $\ell_K:=Y\cap K$. Assume that, $HF^*(K,K;H_\epsilon)\neq\{0\}$ for an $\epsilon$-small wrapping Hamiltonian $H_\epsilon$, then $\ell_K$ admits a Reeb chord for any contact form defining the natural co-oriented contact structure on $Y$.  
	 \end{corollary}
	 \begin{proof}
	 	By $\S 12.5$ Theorem $12.14$ and Corollary $12.15$ of \cite{ritter2017monotone}, we have
	 	\[
	 	SH^*(E)\cong \Lambda[x]/\langle x^{m+1-k}-(-k)^kT\rangle.
	 	\]
	 	Moreover, $c_1(TE)\star .$ corresponds to multiplication by $(m+1-k)x$ under the above identification. Note that, as $T$ is invertible in the Novikov field $\Lambda$, it follows that $x$ is invertible and so is $c_1(TE)$. Therefore, $SH_0^*(E)=\{0\}$. By Proposition \ref{MonotoneCriteria}, it follows that $\ell_K$ admits a Reeb chord for $\alpha_0$, the $\textit{Boothby-Wang}$ contact form on $Y$. Now let $\alpha:=f\alpha_0$ such that $f:Y\rightarrow(0, \infty)$ is a smooth function. In the cylinderical end with the Liouville form $r\alpha_0$, the hypersurface $\{r\equiv f(y)\}$ is of contact-type, having $\alpha$ as the induced contact form. As $K$ is cylinderical, we have that $K\cap\{r\equiv f(y)\}=\ell_K$. Thus, using Proposition \ref{MonotoneCriteria}, it follows that $\ell_K$ admits a Reeb chord for $\alpha$.
	 \end{proof}

\begin{remark}
             Of course, the above Corollary now follows from the stronger result, Theorem $1$ of \cite{shelukhin2026orderability}.
         \end{remark}

     \section{Closed Case}
   \subsection{Fukaya Category}  
   We start this section by recollecting the necessary notions needed from \cite{abouzaid2026quantum}.

Following \cite{abouzaid2026quantum} namely Definition $7.14$ and Definition $11.5$, we recall the definition of the undeformed Fukaya category. \\
	
	Let $(X, \omega)$ be a closed symplectic manifold and fix a base field $\mathbb F=\mathbb R$ or $\mathbb C$. Denote by $\Lambda_0$ the universal Novikov ring over $\mathbb F$ and fix a background cohomology class $st\in H^2(X; \mathbb Z/2)$.
    
    Let, $L_1, \dots, L_k$
	be a finite collection of closed, connected, immersed with transverse self-intersection, Lagrangian submanifolds of $X$ such that:
	\begin{enumerate}
		\item for all $i=1,\dots, k$, $L_i$ is $st$-relatively-spin, as in Definition $9.23$ of \cite{abouzaid2026quantum}, and
		\item for $i\neq j=1, \dots, k$, $L_i$ and $L_j$ intersect transversally. 
	\end{enumerate}
    \begin{defn}[Weak Bounding Cochain]
        For $i=1,\dots, k$ and after writing $L\equiv L_i$, an undeformed weak bounding cochain is
	\[
	b=b_0+b_+\in\Omega(L; \mathbb F)\hat{\otimes}_\mathbb F\Lambda_0
	\]
	such that:
	\begin{enumerate}
		\item The valuation-zero part, $b_0\in\Omega^1(L; \mathbb F)$ is a closed $1$-form, representing the holonomy of a flat $\mathbb F$-line bundle over $L$,
		\item the positive-valuation part $b_+$, is an element of 
		\[
		\Omega^1(L;\mathbb F)\hat{\otimes}_\mathbb F\Lambda_+\oplus\bigoplus_{r\geq 1}\Omega^{2r+1}(L;\mathbb{F})\hat{\otimes}_\mathbb F\Lambda_0
		\]
		\item Using Definition $7.10$, Equations $(7.14)$ and $(7.15)$ of \cite{abouzaid2026quantum}, we require 
		\[
		\sum_{k\geq0}m_k^{b_0}(b_+,\dots, b_+)=\lambda e_L,
		\] 
		where $\lambda\in\Lambda_+$ and $e_L\in\Omega^0(L; \mathbb F)$ the unit constant function on $L$. 
	\end{enumerate}
    \end{defn}
	We call the tuple $\mathbb L=(L, b=b_0+b_+)$ a Lagrangian brane, keeping the choice of $st$-relative spin structure implicit.

As for the $A_\infty$-structure and following \cite{abouzaid2026quantum}, we will be using the $\textit{deRham model}$ (c.f. \cite{fukaya2010lagrangian}, \cite{fukaya2011lagrangian}). Let $\mathcal L:=\{\mathbb L_1, \dots, \mathbb L_k\}$ be a finite collection of Lagrangian branes and for $i=1,\dots, k$, let $\iota_i:L_i\looparrowright X$ be a Lagrangian immersion representing $L_i$. As per \cite{abouzaid2026quantum} Equations $(9.20), (9.21), (9.25)$, we set 
\[
	CF(L_i,L_{j};\mathbb F)
	=
	\begin{cases}
		\displaystyle
		\bigoplus_{x\in \widetilde L_i\times_X\widetilde L_{j}}
		\mathbb F[x],
		&
		L_i\neq L_j,\\[12pt]
		\displaystyle
		\Omega^*(\widetilde L_i;\mathbb F)
		\oplus
		\bigoplus_{x\in
			(\widetilde L_i\times_X\widetilde L_j)
			\setminus\Delta_{\widetilde L_i}}
		\mathbb F[x],
		&
		L_i=L_j.
	\end{cases}
	\]
    Following Definition $11.5$ of \cite{abouzaid2026quantum}, the morphism space in the Fukaya category $\mathcal A$ associated to $\mathcal L$, is given by
    \[
    \hom^{\mathcal A}(\mathbb L_i, \mathbb L_j)\equiv CF(L_i, L_j;\mathbb F)\hat{\otimes}_\mathbb F\Lambda_0,
    \]
for $\mathbb L_i\neq\mathbb L_j$. Moreover, we set 
\[
\hom^{\mathcal A}(\mathbb L_i, \mathbb L_i):=(CF(L_i, L_i;\mathbb F)\hat{\otimes}_\mathbb F\Lambda_0)\oplus\Lambda_0 e\oplus\Lambda_0 f,
\]
where $e$ denotes the homotopy unit and $f$ denotes a degree $-1$ homotopy between $e$ and $e_{L_i}$ representing the constant unit function on $L_i$ (c.f. Definition $9.5$ and Theorem $9.9$ of \cite{abouzaid2026quantum}).

The $A_\infty$-operations are given as follows:\\
For a given $x_i\in\hom^{\mathcal A}(\mathbb L_{s_{i-1}}, \mathbb L_{s_i})$, we set
\[
m_k^{b}(x_1, \dots, x_k):=\sum_{r_0,\dots,r_k\geq0}(-1)^*m^{b_0}_{k+r_0+\dots+r_k}(b_{+, s_0}^{\otimes r_0},x_1, b_{+, s_1}^{\otimes r_1}x_2,\dots,x_k, b_{+, s_{k}}^{\otimes r_k}),
\]
where $*$ is as in $\S 21$ Equation $(21.3)$ and $b_0=(b_{0, s_0}, \dots, b_{0, s_k})$; with the understanding that if $L_i=L_j$ and $b_{0,i}\neq b_{0,j}$, we set
\[
m_{1, 0}(\eta)\equiv d_{dR}(\eta)+(b_{0, j}-b_{0, i})\wedge\eta,
\]
for any $\eta\in\Omega^*(L_i;\mathbb F)$. The above series converges in the $T$-adic topology giving $\mathcal A$ the structure of a filtered, homotopically unital, cyclic $A_\infty$-category.

\begin{remark}
All the following statements and proofs can be extended to the case of $\textit{bulk-deformed}$ Fukaya categories as in \cite{abouzaid2026quantum}. As it is a matter of notational change and to avoid notational clutter, we only deal with the undeformed case. Nevertheless, incorporating bulk-deformation would give us a stronger bound in Proposition \ref{quantitativeresult} and stronger results in general, as there exists Lagrangians which are unobstructed only after bulk-deformation (c.f. \cite{xiao2024moment}).
\end{remark}

\subsection{Mohnke's Construction}
In order to state and prove our desired result, we remind the reader of $\textit{Mohnke's}$ construction as found in \cite{mohnke2001holomorphic}. Let $(V, \alpha)$ be a co-oriented contact manifold and $\ell\subset V$ a closed relatively-spin Legendrian. Let $S<0$ be a real number and consider the symplectization of $(V, \alpha)$ given by $(S, 0]_s\times V$ with the symplectic form $\omega:=d(e^s\alpha)$. Now suppose that, $\ell$ admits no Reeb chords for $\alpha$. For any $T>0$ real number and a smooth embedded loop $\gamma:S^1_\theta\rightarrow (S, 0)\times (0, T)$. We define $F_\gamma:\ell\times S^1_\theta\rightarrow (S, 0]\times V$ by
	 \[
	 F_\gamma(x, \theta):=(s(\theta), \phi^{t(\theta)}_{R_\alpha}(x)),
	 \]
	 where $\gamma(\theta)=(s(\theta), t(\theta))$ and $\phi^t_{R_\alpha}$ denotes the time-$t$ flow of the Reeb vectorfield $R_\alpha$ associated to $\alpha$. Set $L_\gamma:=F_\gamma(\ell)\subset (S, 0]\times V$ and notice that, $F^*_\gamma(e^s\alpha)=e^{s(\theta)}t'(\theta)d\theta$ and hence, $L_\gamma$ is a Lagrangian of $((S, 0], d(e^s\alpha))$. On the other hand, under the assumption that $\ell$ admits no Reeb chords for $\alpha$, it follows that $F_\gamma$ is an embedding and hence $L_\gamma$ is an embedded relatively-spin Lagrangian.

\begin{remark}
    Strictly speaking, we do not need to assume a relative-spin structure on the Legendrian to apply Mohnke's construction. We assume so, as it is a technical assumption that we need in what follows. For this reason, we assume without stating, that all Legendrians are $st$-relatively-spin.
\end{remark}

\subsection{Existence Results}
Now let $(X,\omega)$ be a general closed symplectic manifold. Suppose that, for $i=1,\dots, k$, $(L_i, b_i)$ is a pair of immersed, relatively-spin, weakly unobstructed Lagrangians, together with a boundaing cochain $b_i\in CF^{\text{odd}}(L_i; \Lambda)$ with the understanding that $b_i\equiv b_{0,i}+b_{+, i}$; such that, $\mathcal A:=\langle(L_1, b_1),\dots, (L_k, b_k)\rangle$ satisfies the split-generation criteria as in \cite{abouzaid2026quantum}. Namely, and following the notation and terminology as in \cite{abouzaid2026quantum}, we assume that the open-closed map defined on the full $A_\infty$-subcategory $\mathcal A$ satisfies
	 \[
	 e_X\in\operatorname{Im}(\hat{p}_\mathcal{A}:HH_*(\mathcal A, \Lambda)\rightarrow QH^*(X;\Lambda)).
	 \] 
	 That is, there exists $[\xi]\in HH_*(\mathcal{A}; \Lambda)$ such that, $\hat{p}_{\mathcal A}([\xi])=e_X\in QH^*(X; \Lambda)$. Moreover, following the constructions as in $\S 24, 32$ of \cite{abouzaid2026quantum} and using Proposition $24.1$ and Theorem $32.1$ of \cite{abouzaid2026quantum}, it follows that, there exists:
	 \begin{enumerate}
	 	\item an integer $D\geq0$, thought of as a $\textit{word-length}$ bound and
	 	\item a real number $\rho\geq0$, thought of as an $\textit{energy-shift}$ factor
	 \end{enumerate}
	 such that, for some $E>0$, in particular $E>\rho$, the truncated open-closed map 
	 \[
	 \hat{p}^E_\mathcal A:HH_*^{\leq D}(\mathcal A; \Lambda_0/T^E\Lambda_0)\rightarrow QH^*(X; \Lambda_0/T^E\Lambda_0),
	 \]
	 associated to $\mathcal{A}$, satisfies 
	 \[
	 \hat{p}^E_\mathcal A([\xi_E])=T^\rho e_X\in QH^*(X, \Lambda_0/T^E\Lambda_0),
	 \]
	 for some $[\xi_E]\in HH_*^{\leq D}(\mathcal A; \Lambda_0/T^E\Lambda_0)$. We will refer to such conditions as the \\
     $\textbf{truncated generation-criteria}$.\\
     
	 Now let $(V, \alpha)$ be a connected hypersurface of $X$ of contact-type such that for $i=1,\dots, k$, $V$ intersects each of $L_i$ transversally and in a Legendrian $\ell_i$.
	 \begin{proposition} \label{MainResClosedCase}
	 	Let $\ell$ be a closed Legendrian submanifold of $(V, \alpha)$ such that, there exists an $E>\rho$, where $\rho$ is the energy-shift factor, as in the $\textbf{truncated generation-criteria}$, there exists a $T>0$ and a smooth embedded loop $\gamma:S^1_\theta\rightarrow (S, 0)\times(0, T)$ such that
	 	\begin{enumerate}
	 		\item $L_\gamma$ is weakly unobstructed modulo $T^E$. That is, there exists $b_\gamma\in CF^{\text{odd}}(L_\gamma; \Lambda_+/T^E\Lambda_0)$ such that,
	 		\[
	 		\sum_{k\geq 0}m_k(b_\gamma, \dots, b_\gamma)\equiv \lambda e_{L_\gamma} \text{     mod }T^E\Lambda_0.
	 		\]
	 		\item $T^\rho e_{L_\gamma}\neq 0\in HF^*((L_\gamma, b_\gamma), (L_\gamma, b_\gamma); \Lambda_0/T^E\Lambda_0)$.
	 	\end{enumerate}
	 	Then, either $\ell$ admits a self-Reeb chord or a Reeb chord to the Legendrian link $\bigcup_i\ell_i$, for the contact form $\alpha$.
	 \end{proposition}
	 
	 \begin{proof}
	 	Let $K$ be any closed embedded relatively-spin Lagrangian submanifold of $X$ and assume that $K$ is weakly unobstructed modulo $T^E$. Observe that, there exists an $i=1, \dots, k$ such that $K$ intersects $L_i$ transversally, possibly after a small Hamiltonian perturbation. Indeed, otherwise 
	 	\[
	 	CF^*(L_i, K; \Lambda)=CF^*(K, L_i; \Lambda)=\{0\},
	 	\] 
	 	for any $i=1,\dots, k$. Using $\S23$ Equations $(23.2)$-$(23.4)$ of \cite{abouzaid2026quantum}, it follows that the two-sided Yoneda module
	 	\[
	 	{}_{K}\mathcal{F}_\mathcal A\hat{\otimes}_\mathcal A{}_\mathcal A\mathcal F_{K}
	 	\]
	 	is trivial. Using Equation $(32.4)$ of \cite{abouzaid2026quantum}, it follows that, $\hat{q}_{K}\circ\hat{p}^E_\mathcal A\equiv 0$, where $\hat{q}_K$ is the associated closed-open to $K$. Using unitality of $\hat{q}_K$ and as $\hat{p}^E_\mathcal A([\xi_E])=T^\rho e_X$, we get a contradiction. Therefore, for some $i=1,\dots, k$, $K$ intersects $L_i$ transversally, possible after a small Hamiltonian perturbation.
	 	
	 	Now let $S<0$ and $\iota:(S, 0]_s\times V\hookrightarrow X$ be a symplectic embedding satisfying $\iota^*\omega=d(e^s\alpha)$ and identifying $\{0\}\times V$ with $V$. Note that, possibly after an exact perturbation of $d(e^s\alpha)$ supported away from $\{0\}\times V$ and taking a bigger $S<0$, we may assume that 
	 	\[
	 	L_i\cap(\iota((S, 0]\times V))\equiv \iota((S, 0]\times\ell_i),
	 	\]
	 	for every $i=1, \dots, k$. Let $T>0$ and $\gamma:S^1\rightarrow (S, 0)\times (0, T)$ be as in the statement. Following the above discussion, we denote by $L_\gamma$ the associated Lagrangian submanifold of $X$, given by Mohnke's construction on $\ell$ using $\gamma$ and $\iota$. If $\ell$ admits no Reeb chords for $\alpha$, then applying the above argument for $K\equiv L_\gamma$, we get a contradiction, which completes the proof.
	 \end{proof}

\begin{remark}
The argument used in the proof of Proposition~3 suggests a further
application of the generation-criteria of \cite{abouzaid2026quantum}, which we leave for future work. Choose a contact-type collar of
$V$ on which the almost-complex structure is cylindrical and which is
disjoint from the support of the perturbation data used in the construction of the Cardy diagram as in \cite{abouzaid2026quantum}. One may then stretch the neck along $V$ in the moduli spaces of pseudo-holomorphic annuli entering that diagram, having an interior marked point constraint on $V$. Any non-trivial symplectization level should contain either an interior puncture asymptotic to a closed Reeb orbit or a boundary puncture asymptotic to a Reeb chord of the Legendrian link $\bigcup_i\ell_i$. Therefore, either $V$ carries a closed Reeb orbit or $\bigcup_i\ell_i$ carries a Reeb chord. A complete proof would require an SFT compactness and gluing argument for these annulus moduli spaces, compatible with the perturbations used to establish the Cardy relation as in \cite{abouzaid2026quantum}.
\end{remark}

\subsection{Quantitative Crietria}
	 Under the assumption that $(V, \alpha)\subset (X, \omega)$ is a hypersurface of contact-type, it follows that $\omega|_V=d\alpha$. In particular, and using the deRham model of relative cohomology, it follows that the pair $(\omega, \alpha)$ defines a non-zero class $[\omega, \alpha]\in H^2(X, V;\mathbb R)$. Let $P:\pi_2(X,V)\rightarrow\mathbb R$ be the $(\omega, \alpha)$-associated group morphism, given by
	 \[
	 P([u]):=\int_{\mathbb D}u^*\omega-\int_{\partial\mathbb D}u^*\alpha,
	 \]
	 where $u:(\mathbb D, \partial\mathbb D)\rightarrow(X,V)$ is a continuous map. Using the notation as in $\textit{Mohnke's}$ construction discusion above, we set $C_\gamma:=\int_\gamma e^sdt$.
     
	 \begin{defn} \label{QuanCriteria}
	 	We define $\delta_\gamma:=\inf((C_\gamma\mathbb Z+P(\pi_2(X,V))\cap\mathbb{R}_{>0})\in[0, \infty]$, with the understanding that if $(C_\gamma\mathbb Z+P(\pi_2(X,V))\cap\mathbb{R}_{>0}=\emptyset$ then we set $\delta_\gamma\equiv\infty$.
	 \end{defn}

	 \begin{proposition}\label{quantitativeresult}
	 	Let $\ell\subset V$ be a closed Legendrian and suppose that, there exists a smooth loop $\gamma$ such that $\delta_\gamma>\rho$. Then, either $\ell$ admits a self-Reeb chord or a Reeb chord joing $\ell$ to the Legendrian link $\bigcup_i\ell_i$, for the contact form $\alpha$.
	 \end{proposition}
	 
	 \begin{proof}
	 	Following our notations and the constructions as above, let $L_\gamma$ be the associated Lagrangian given by $\ell$ and $\gamma$. It suffice to show that $L_\gamma$ satisfies the conditions of Proposition \ref{MainResClosedCase}. To this end, assume that $\ell$ admits no Reeb chords for $\alpha$ and let $\beta\in\pi_2(X, L_\gamma)$. Denote by $\kappa(\beta)\in\mathbb Z$ the winding number of $\partial\beta$ around the $S^1$-factor under the identification $L_\gamma\cong S^1\times\ell$. Denote by $\bar{\beta}\in\pi_2(X, V)$ given by attaching the $\textit{radial collar annulus}$ of $\partial\beta$ to its projection onto the $V$-factor. By Stokes' Theorem, it follows that
	 	\[
	 	\omega(\beta)-\kappa(\beta)C_\gamma=P(\bar{\beta}).
	 	\]
	 	Therefore, $\omega(\pi_2(X, L_\gamma))\subseteq C_\gamma\mathbb Z+P(\pi_2(X, V))$. In particular, for a $u:(\mathbb D, \partial\mathbb D)\rightarrow (X, L_\gamma)$ non-constant pseudo-holomorphic curve, its energy satisfies $E(u)\geq\delta_\gamma$.
	 	
	 	By assumption, $\delta_\gamma>\rho$ and hence for $\rho<E<\delta_\gamma$ and after writing the $A_\infty$-operations on $\Omega^*(L_\gamma;\mathbb F)\hat{\otimes}_{\mathbb F}\Lambda_0$ as $m_k=m_{k, 0}+\sum_{\omega(\beta)>0}m_{k,\beta}T^{\omega(\beta)}$, it follows that, (c.f. Equation $6.13$ of \cite{abouzaid2026quantum})
	 	\[
	 	m_{1, 0}=d_{dR}, 
	 	\]
	 	\[
	 	m_{1, \beta}(\eta)=d_{dR}(\eta)\text{     mod } T^E\Lambda_0.
	 	\]
	 	In particular, $L_\gamma$ is unobstructed modulo $T^E$. Now consider $\Omega^*(L_\gamma;\mathbb F)\hat{\otimes}_{\mathbb F}\Lambda_0/T^E\Lambda_0$ with its $T$-adic topology given by the $T$-filtration which we denote by $\mathcal{F}$. By Theorem $D$ and Theorem $6.3.28$ of \cite{Fukaya2009} or Theorem $32.1$ of \cite{abouzaid2026quantum}, the $\mathcal{F}$-filtration on
	 	\[
	 	(\Omega^*(L_\gamma;\mathbb F)\hat{\otimes}_{\mathbb F}\Lambda_0/T^E\Lambda_0, m_1 \text{   mod }T^E)
	 	\]
	 	induces a convergent spectral sequence whose $E_2$-page is of the form
	 	\[
	 	E_2^{*,*}\cong H^*(L_\gamma;\mathbb F)\hat{\otimes}_\mathbb F gr_\mathcal{F}\Lambda_0/T^E\Lambda_0.
	 	\] 
	 	As $m_1(\eta)\equiv d_{dR}(\eta)\text{   mod }T^E\Lambda_0$ for any $\eta\in\Omega^*(L_\gamma;\mathbb F)\hat{\otimes}_{\mathbb F}\Lambda_0/T^E\Lambda_0$, it follows that the differentials of the spectral sequence $d_r$ vanishes for $r\geq2$. In particular, the spectral sequence collapse at its $E_2$-page and as the $\mathcal F$-filtration on $\Lambda_0/T^E\Lambda_0$ is bounded, exhaustive and complete, it follows that,
	 	\[
	 	HF^*(L_\gamma, L_\gamma; \Lambda_0/T^E\Lambda_0)\cong H_{dR}(L_\gamma)\hat{\otimes}_\mathbb R\Lambda_0/T^E\Lambda_0.
	 	\]
	 	Note that, the class $T^\rho e_{L_\gamma}$ is a cocycle as $e_{L_\gamma}\equiv 1_{L_\gamma}\in\Omega^0(L_\gamma)$ and $\rho<E$. Thus, $T^\rho e_{L_\gamma}\neq 0\in H^0_{dR}(L_\gamma)\hat{\otimes}_\mathbb R gr_\mathcal F \Lambda_0/T^E\Lambda_0$. As the spectral sequence collapses at its $E_2$-page, it follows that, $T^\rho e_{L_\gamma}\neq 0\in HF^*(L_\gamma, L_\gamma; \Lambda_0/T^E\Lambda_0)$. Now the result follows using Proposition \ref{MainResClosedCase}.
	 \end{proof}

     \begin{remark}
         Of course, one can apply the above result to the case when $\ell\equiv\ell_i$ for some $i=1,\dots, k$ and deduce that, under the setting of the above Proposition and for any $i=1,\dots, k$, there exists a $j=1,\dots k$ and at least one Reeb chord for the contact form $\alpha$, joining $\ell_i$ to $\ell_j$. 
     \end{remark}
\begin{example}
   Let $X$ be $\mathbb{CP}^2$ with its Fubini-Study symplectic form $\omega$, normalized so that $\int_{\mathbb {CP}^1}\omega\equiv 1$. Following the notations as in $\S 28$ of \cite{abouzaid2026quantum}, let $L\equiv L(\frac{1}{3}, \frac{1}{3})$ be the Clifford torus, with its weak bounding cochains $b_0, b_1, b_2$ and set $\mathcal A:=\langle (L, b_0), (L, b_1), (L, b_2)\rangle$. Note that, using $\S 32.2$ of \cite{abouzaid2026quantum}, it follows that $T^{\frac{2}{3}}e_{\mathbb {CP}^2}\in\operatorname{Im} \hat{p}_{\mathcal A}$. Let $V\subset X$ be the hypersurface of contact-type given by $V:=\mu^{-1}(\{(u_1, u_2):u_1+u_2=\frac{2}{3}\})$, where the is the standard moment map $\mu:\mathbb {CP}^2\rightarrow\{(u_1, u_2)\in\mathbb R^2:u_1, u_2\geq0, u_1+u_2\leq1\}$ given by $\mu([z_0:z_1:z_2])=\left(\frac{|z_1|^2}{\sum_i|z_i|^2}, \frac{|z_2|^2}{\sum_i|z_i|^2}\right)$. Notice that, $V\cong S^3$ and $L\subset V$. On the other hand, after a small Hamiltonian perturbation of $L$ denoted by $L^\epsilon$, we have that $L^\epsilon$ intersects $V$ transversally in two Legendrians, $\ell_1$ and $\ell_2$. Now let $\ell$ be any closed relatively-spin Legendrian of $V$ and following the same notation as above, let $\gamma$ be a loop so that $C_\gamma\equiv1$. As $V\cong S^3$, it follows that $\delta_\gamma=1>\frac{2}{3}\equiv\rho$. Therefore by Proposition \ref{quantitativeresult}, it follows that, either $\ell$ admits a self-Reeb chord or a Reeb chord joining it to $\ell_1\sqcup \ell_2$. Moreover, any collar neighborhood of $V$ is non-displaceable in $\mathbb {CP}^2$ as otherwise, the Clifford torus would be Hamiltonianly displaceable in $\mathbb{CP}^2$. 
\end{example}

  To finish this section, we relate our quantitative constant $\delta_\gamma$ as in Definition \ref{QuanCriteria} to other intrinsic quantitative constants as in \cite{cieliebak2018punctured} and \cite{Zhou2020OnTM}.
  \begin{defn}[\cite{cieliebak2018punctured}]
    For a Lagrangian $L$ of $X$, we define
    \[
    A_{min}(L):=\inf\{\omega(\beta)>0:\beta\in\pi_2(X,L)\}.
    \]
  \end{defn}
 Using the same notation, we observe that, if $\omega(\pi_2(X, L_\gamma))\subseteq C_\gamma\mathbb Z+P(\pi_2(X,V))$ 
 then, $\delta_\gamma\leq A_{min}(L_\gamma)$. In particular, we have $\delta_\gamma\leq A_{min}(L_\gamma)\leq C_\gamma$. For instance, if $(X, \omega)$ is symplectically aspherical and $\pi_1(V)\hookrightarrow\pi_1(X)$ is an injective group morphism, then $\delta_\gamma\equiv C_\gamma\equiv A_{min}(L_\gamma)$.

\begin{corollary} \label{MohnkeResult}
    Under the same setting as above and assuming that $(X, \omega)$ is symplectically aspherical and that $\pi_1(V)\hookrightarrow\pi_1(X)$ is an injective group morphism; any closed Legendrian $\ell$ of $V$ admits either a self-Reeb chord for the contact form $\alpha$ or a Reeb chord for $\alpha$, joining $\ell$ to the Legendrian link $\bigcup_i\ell_i$.
\end{corollary}
The above Corollary is closely related to the well-known result of Mohnke \cite{mohnke2001holomorphic}, yet we do not impose any Hamiltonian-displaceability assumption. We illustrate such difference in the following result. 

\begin{theorem}\label{cosphereResult}
    For $n\geq3$, let $V=S^*\mathbb T^n$ be the unit cosphere-bundle of the flat torus $\mathbb T^n$ together with its canonical contact form $\alpha$ induced from $T^*\mathbb T^n$. There exists a closed embedded Legendrian link $\ell'$ such that, every closed spin Legendrian $\ell$ of $V$ either admits a non-constant self-Reeb chord for $\alpha$ or a Reeb chord for $\alpha$ joining $\ell$ to $\ell'$.
\end{theorem}

\begin{proof}
    Let $X=\mathbb T^{2n}\cong\mathbb T^n\times \mathbb T^n$ together with its standard symplectic form and $V=S^*\mathbb T^n$ be the unit cosphere bundle, thought of as the boundary of a small enough Weisntein neighborhood of $\mathbb T^n\times\{1\}\subset X$. Denote by $\mathcal A_1:=\langle(l, 0), (m,0)\rangle$ the $A_\infty$-category generated by the the longitude and meridian circles of the torus $\mathbb T=S^1\times S^1$ with zero bounding cochains. Using $\S 6.3$ of \cite{abouzaid2010homological}, namely Theorem $6.4$ and Remark $6.5$ of \cite{abouzaid2010homological}, it follows that, $\mathcal A_1$ is a cohomologically smooth $A_\infty$-category. In particular, and following the notation as in $\S 26$ of \cite{abouzaid2026quantum}, it follows that the $n$-fold algebraic tensor-product $A_\infty$-category $\mathcal A^{\otimes n}\equiv\mathcal A_1\otimes\dots\otimes\mathcal A_1$ is cohomologically smooth. On the other hand, as the geometric-product $A_\infty$-category, $\mathcal A^n\equiv\langle (L_1\times \dots\times L_n,0)\rangle$ such that $L_i\in\{l, m\}$ for $i=1,\dots,n$, has the same objects as $\mathcal A^{\otimes n}$, it follows by Theorem $26.3$, Equation $(26.20)$ of \cite{abouzaid2026quantum} (c.f. Theorem $16.9$ of \cite{fukaya2025unobstructed}) that $\mathcal A^n$ is cohomologically smooth. Note that, $c_1(X)=0$ and hence, by Remark $23.6$ of \cite{abouzaid2026quantum} (c.f. \cite{ganatra2016automatically}) and Theorem $1.8$ of \cite{abouzaid2026quantum}, it follows that $\mathcal A^n$ satisfies the generation-criteria of \cite{abouzaid2026quantum}. After a small Hamiltonian perturbation of $L_1\times\dots\times L_n$ and abuse of notation, we may assume that $\ell'
:=
\bigsqcup_{(L_1,\ldots,L_n)\in\{l,m\}^n}
\left((L_1\times\cdots\times L_n)\cap V\right)$ is a Legendrian link of $V$. Now the result follows by Corollary \ref{MohnkeResult}.
\end{proof}

\bibliography{ref}
\bibliographystyle{alpha}
\bigskip
\noindent Emails:
\texttt{mohamad@amss.ac.cn},
\texttt{mohammadbrabah@gmail.com}
\end{document}